\documentclass[12pt]{amsart}
\usepackage{amsmath,amsthm,amsfonts,amssymb,amscd,euscript}

\newlength{\defbaselineskip}
\theoremstyle{plain}

\theoremstyle{definition}

\newcommand{\gchoose}[2]{\left[\begin{array}{c}{#1 } \\{#2} \end{array}\right]  }

\theoremstyle{plain}
\newtheorem{thm}{Theorem}

\newtheorem{lem}[thm]{Lemma}
\newtheorem{cor}[thm]{Corollary}

\newtheorem{prop}[thm]{Proposition}

\theoremstyle{definition}
\newtheorem{defn}[thm]{Definition}

\newtheorem{rem}[thm]{Remark}
\numberwithin{equation}{section}
\begin{document}





\title[Rank two acyclic cluster algebrass]{On cluster variables of rank two acyclic cluster algebras}
\author{Kyungyong Lee}
\thanks{Research partially supported by NSF grant DMS 0901367.}


\address{Department of Mathematics, University of Connecticut, Storrs, CT 06269}
\email{{\tt kyungl@purdue.edu}}
\begin{abstract}
In this note, we find an explicit formula for the Laurent expression of  cluster variables of coefficient-free rank two cluster algebras associated with the matrix $\left(\begin{array}{cc} 0 & c\\-c & 0   \end{array}\right)$, and show that a large number of coefficients are non-negative.  As a corollary, we obtain an explicit expression for the Euler-Poincar\'{e} characteristics of the corresponding quiver Grassmannians.
\end{abstract}

 \maketitle
 



\section{introduction}

Let $b,c$ be positive integers and $x_1, x_2$ be indeterminates. The (coefficient-free) \emph{cluster algebra} $\mathcal{A}(b,c)$ is the subring of the field $\mathbb{Q}(x_1,x_2)$ generated by the elements $x_m$, $m\in \mathbb{Z}$ satisfying the recurrence relations:
$$
x_{n+1} =\left\{ \begin{array}{cl}
(x_n^b +1)/{x_{n-1}} & \text{ if } n \text{ is odd,} \\
\text{ } & \text{ }\\
{(x_n^c +1)}/{x_{n-1}} & \text{ if } n \text{ is even.}
\end{array}      \right.
$$ 
The elements $x_m$, $m\in \mathbb{Z}$ are called the cluster variables of $\mathcal{A}(b,c)$.  Fomin and Zelevinsky \cite{FZ} introduced cluster algebras and proved the Laurent phenomenon whose special case says that for every $m\in  \mathbb{Z}$ the cluster variable $x_m$ can be expressed as a Laurent polynomial of $x_1^{\pm 1}$ and  $x_2^{\pm 1}$. In addition, they conjectured that the coefficients of monomials in the Laurent expression of $x_m$ are non-negative integers.  When $bc\leq 4$, Sherman-Zelevinsky \cite{SZ} and independently Musiker-Propp \cite{MP} proved the conjecture. Moreover in this case the explicit combinatorial formulas for the coefficients are known.  In this paper, we find an explicit formula for the coefficients when $b=c\geq 2$, and show that a large number of coefficients  are non-negative.

As we will frequently use product forms, we say a few words about our convention. When we have any  integer $A$   and any function $f(i)$ of $i$, the product $\prod_{i=A}^{A-1} f(i)$ will be defined to be 1.

Before we state our main results, we need some definitions.

\begin{defn}\label{modifiedbinomialcoeff}
For arbitrary (possibly negative) integers $A, B$, we define the modified binomial coefficient as follows.
$$\left[\begin{array}{c}{A } \\{B} \end{array}\right] := \left\{ \begin{array}{ll}  \prod_{i=0}^{A-B-1} \frac{A-i}{A-B-i}, & \text{ if }A > B\\ \, & \,  \\   1, & \text{ if }A=B \\ \, & \, \\  0, & \text{ if }A<B.  \end{array}  \right.$$ \qed
\end{defn} 

If $A\geq 0$ then $\left[\begin{array}{c}{A } \\{B} \end{array}\right]=\gchoose{A}{A-B}$ is just the usual binomial coefficient. In general, $\gchoose{A}{A-B}$ is equal to the generalized binomial coefficient ${A \choose B}$. But in this paper we use our modified binomial coefficients to  avoid too complicated expressions.

\begin{defn}
Let $\{a_n\}$ be the sequence  defined by the recurrence relation $$a_n=ca_{n-1} -a_{n-2},$$ with the initial condition $a_1=0$, $a_2=1$. If $c=2$ then $a_n=n-1$. When $c>2$, it is easy to see that 
$$
a_n= \frac{1}{\sqrt{c^2-4}  }\left(\frac{c+\sqrt{c^2-4}}{2}\right)^{n-1} - \frac{1}{\sqrt{c^2-4}  }\left(\frac{c-\sqrt{c^2-4}}{2}\right)^{n-1} = \sum_{i\geq 0} (-1)^i { {n-2-i} \choose i }c^{n-2-2i}.
$$ \qed
\end{defn}

\begin{rem}
It is easy to show that for any $n$,
\begin{equation}\label{negone}
a_{n-1}a_{n-3}-a_{n-2}^2=-1,
\end{equation}
which we will use later.
\end{rem}

Our main result is the following.

\begin{thm}\label{mainthm} Assume that $b=c\geq 2$. Let $n\geq 3$. Then
\begin{equation}\label{mainformula}\aligned
&x_n= x_1^{-a_{n-1}} x_2^{-a_{n-2}} \sum_{e_1,e_2} \sum_{ t_0,t_1,\cdots,t_{n-4}} \left[ \left( \prod_{i=0}^{n-4} \left[\begin{array}{c}{{a_{i+1} - cs_i    } } \\{t_i} \end{array}\right]   \right) \right.\\
&\,\,\,\,\,\,\,\,\,\,\,\,\,\times \left.\gchoose{a_{n-2} - cs_{n-3} }{a_{n-2} - cs_{n-3}-e_2+s_{n-4}}\gchoose{-a_{n-3} + c e_2 }{  -a_{n-3} + c e_2-e_1+s_{n-3} } x_1^ {c(a_{n-2}-e_{2})} x_2^{ce_{1}}\right],
\endaligned\end{equation}
where $$
s_i=\sum_{j=0}^{i-1} a_{i-j+1}t_j,
$$
and the summations run over all integers $e_1,e_2, t_0,...,t_{n-4}$ satisfying 
\begin{equation}\label{cond501}\left\{
\begin{array}{l} 0\leq t_i \leq a_{i+1} - cs_i \, (0\leq i\leq n-4),\\
 0\leq a_{n-2} - cs_{n-3}-e_2+s_{n-4}\leq a_{n-2} - cs_{n-3}, \text{ and }  \\
e_2 a_{n-1} -e_1 a_{n-2}\geq 0.      
 \end{array} \right.\end{equation}
\end{thm}

Since $\gchoose{A}{B}\neq 0$ if and only if $A\geq B$, we may add the condition $0\geq -e_1+s_{n-3}$ to (\ref{cond501}). Then the summation in the statement is guaranteed to be a finite sum. A referee remarks that $F$-polynomials have similar expressions.  As he pointed out, the expression without (\ref{cond501}) is an easy consequence of the  formula (6.28) in the paper \cite{FZ4} by Fomin and Zelevinsky, and the one with $e_2 a_{n-1} -e_1 a_{n-2}\geq 0$  is a consequence of \cite[Proposition 3.5]{SZ} in the paper by Sherman and Zelevinsky. Our contribution is to show that  all the modified binomial coefficients in (\ref{mainformula}) except for the last one are non-negative. 

As a corollary to Theorem~\ref{mainthm}, we obtain  an expression for the Euler-Poincar\'{e} characteristic of  the variety $\text{Gr}_{(e_1,e_2)}(M(n))$ of all subrepresentations of dimension $(e_1,e_2)$ in a unique (up to an isomorphism) indecomposable $Q_c$-representation $M(n)$ of dimension $(a_{n-1}, a_{n-2})$, where  $Q_c$ is the generalized Kronecker  quiver with  two vertices 1 and 2, and $c$ arrows from 1 to 2. We use a result of Caldero and  Zelevinsky \cite[Theorem 3.2 and (3.5)]{CZ}.

\begin{thm}[Caldero and  Zelevinsky]
The cluster variable $x_n$ is equal to 
$$
x_1^{-a_{n-1}} x_2^{-a_{n-2}}\sum_{e_1,e_2} \chi(\emph{Gr}_{(e_1,e_2)}(M(n)))x_1^ {c(a_{n-2}-e_{2})} x_2^{ce_{1}}.
$$
\end{thm}

\begin{cor}\label{maincor}
Assume that $b=c\geq 2$. For any $(e_1,e_2)$ and $n\geq 3$, the Euler-Poincar\'{e} characteristic of $\emph{Gr}_{(e_1,e_2)}(M(n))$ is equal to
\begin{equation}\label{submainformula}\sum_{ t_0,t_1,\cdots,t_{n-4}} \left[ \left( \prod_{i=0}^{n-4} \left[\begin{array}{c}{{a_{i+1} - cs_i    } } \\{t_i} \end{array}\right]   \right)\gchoose{a_{n-2} - cs_{n-3} }{a_{n-2} - cs_{n-3}-e_2+s_{n-4}}\gchoose{-a_{n-3} + c e_2 }{  -a_{n-3} + c e_2-e_1+s_{n-3} } \right],   
\end{equation}
where  the summation runs over all integers $t_0,...,t_{n-4}$ satisfying 
\begin{equation}\label{cond512}\left\{
\begin{array}{l} 0\leq t_i \leq a_{i+1} - cs_i \, (0\leq i\leq n-4), \text{ and } \\
  0\leq a_{n-2} - cs_{n-3}-e_2+s_{n-4}\leq a_{n-2} - cs_{n-3}.      
 \end{array} \right.\end{equation}\end{cor}

\begin{cor}\label{maincor2}
Assume that $b=c\geq 3$. Let $n\geq 3$. For any $(e_1,e_2)$ with $e_2\geq \frac{a_{n-3}}{c}$, the Euler-Poincar\'{e} characteristic of $\emph{Gr}_{(e_1,e_2)}(M(n))$ is non-negative.
\end{cor}

\noindent\emph{Acknowledgement.} We are grateful to Gr\'{e}goire Dupont for valuable discussions and correspondence. We also thank anonymous referees for their useful suggestions and helpful comments.

\section{Proofs}

We actually prove the following statement, which is equivalent to Theorem~\ref{mainthm} but simpler to prove.
\begin{thm}\label{mainthm2} Assume that $b=c\geq 2$. Let $n\geq 3$. Then
\begin{equation}\label{mainformula2}\aligned
&x_n= x_1^{-a_{n-1}} x_2^{-a_{n-2}} \sum_{ t_0,t_1,\cdots,t_{n-2}} \left[ \left( \prod_{i=0}^{n-2} \left[\begin{array}{c}{{a_{i+1} - cs_i    } } \\{t_i} \end{array}\right]   \right)x_1^ {cs_{n-2}} x_2^{c(a_{n-1}-s_{n-1})}\right],
\endaligned\end{equation}
where $$
s_i=\sum_{j=0}^{i-1} a_{i-j+1}t_j,
$$
and the summation runs over all integers $ t_0,...,t_{n-2}$ satisfying 
\begin{equation}\label{cond502}\left\{
\begin{array}{l} 0\leq t_i \leq a_{i+1} - cs_i \, (0\leq i\leq n-3), \text{ and } \\
s_{n-1} a_{n-2} -s_{n-2} a_{n-1}  \geq 0.      
 \end{array} \right.\end{equation}
\end{thm}
\begin{lem}
Theorem~\ref{mainthm2} is equivalent to Theorem~\ref{mainthm}.
\end{lem}
\begin{proof}
In (\ref{mainformula}), if we substitute $a_{n-i}-s_{n-i}$ for $e_i$ ($i=1,2$), we obtain (\ref{mainformula2}). Note that the coefficient of $t_{n-i-1}$ in $s_{n-i}$ is equal to 1. Hence, as $e_i$ runs over integers, so does $t_{n-i-1}$.
\end{proof}

\begin{proof}[Proof of Theorem~\ref{mainthm2}]
 It is not hard to check the statement for $n=3,4,5$. When $n\geq 5$, we use induction on $n$.

 Suppose that the statement holds for $n$ or less. Then by the obvious shift, we have
\Small{ $$
x_{n+1}= x_2^{-a_{n-1}} x_3^{-a_{n-2}}  \sum_{ t_0,t_1,\cdots,t_{n-2}} \left[ \left( \prod_{i=0}^{n-2} \left[\begin{array}{c}{{a_{i+1} - cs_i    } } \\{t_i} \end{array}\right]   \right)  x_2^ {cs_{n-2}} x_3^{c(a_{n-1} - s_{n-1})}\right],
 $$}\normalsize{where the summation runs over all integers $ t_0,...,t_{n-2}$ satisfying (\ref{cond502}).
 
 Substituting $\frac{x_2^c+1}{x_1}$ into $x_3$, we get}
\Small{
$$\aligned &x_{n+1}\\
&= x_2^{-a_{n-1}}  \sum_{ t_0,t_1,\cdots,t_{n-2}} \left[ \left( \prod_{i=0}^{n-2} \left[\begin{array}{c}{{a_{i+1} - cs_i    } } \\{t_i} \end{array}\right]   \right)  x_2^ {cs_{n-2}} \left(\frac{x_2^c+1}{x_1}\right)^{c(a_{n-1} - s_{n-1})-a_{n-2}}\right],
\\
&= x_2^{-a_{n-1}}  \sum_{ t_0,t_1,\cdots,t_{n-2}} \left[ \left( \prod_{i=0}^{n-2} \left[\begin{array}{c}{{a_{i+1} - cs_i    } } \\{t_i} \end{array}\right]   \right)  x_2^ {cs_{n-2}} \left(\frac{x_2^c+1}{x_1}\right)^{a_{n} - cs_{n-1}}\right],
\\
&= x_2^{-a_{n-1}}  \sum_{ t_0,t_1,\cdots,t_{n-2}} \left[ \left( \prod_{i=0}^{n-2} \left[\begin{array}{c}{{a_{i+1} - cs_i    } } \\{t_i} \end{array}\right]   \right) \sum_{t_{n-1}\in \mathbb{Z}}     \left[\begin{array}{c}{{a_{n} - c s_{n-1}   } } \\{t_{n-1}} \end{array}\right]  (x_2^c)^{a_{n} - cs_{n-1}-t_{n-1}}  {x_1}^{cs_{n-1} - a_{n}}  x_2^ {cs_{n-2}} \right],
\\
&= x_2^{-a_{n-1}}  \sum_{ t_0,t_1,\cdots,t_{n-1}} \left[ \left( \prod_{i=0}^{n-1} \left[\begin{array}{c}{{a_{i+1} - cs_i    } } \\{t_i} \end{array}\right]   \right) (x_2^c)^{a_{n} - cs_{n-1}-t_{n-1}}  {x_1}^{cs_{n-1} - a_{n}}  x_2^ {cs_{n-2}} \right],
\\
&=  x_1^{-a_n} x_2^{-a_{n-1}} \sum_{ t_0,t_1,\cdots,t_{n-1}} \left[ \left( \prod_{i=0}^{n-1} \left[\begin{array}{c}{{a_{i+1} - cs_i    } } \\{t_i} \end{array}\right]   \right) {x_1}^{cs_{n-1}} x_2^{c(a_{n} - s_{n})}    \right],
\endaligned$$}
\normalsize{where  the last equality follows from }
$$s_n=\sum_{j=1}^{n-1} a_{n-j+1}t_j=\sum_{j=1}^{n-1} (ca_{n-j}-a_{n-j-1})t_j=cs_{n-1}-s_{n-2}+t_{n-1}.$$If one worries about convergence of the sum, then we could have begun by assuming that $|x_2|<1$, but since we will eventually show that the sum is finite, the convergence should not be a problem.

Remember that $t_0,...,t_{n-2}$  satisfy (\ref{cond502}). By identifying $a_{n+1-i}-s_{n+1-i}$ with $e_i$ ($i=1,2$), Proposition~\ref{eanegineqcor}  implies that even if $t_{n-2}$ and $t_{n-1}$ run over the only integers satisfying $s_{n} a_{n-1} -s_{n-1} a_{n}  \geq 0$, we get the same result. On the other hand, in order to prove that Theorem~\ref{mainthm2} holds for $n+1$, we need to show that $t_{n-2}$ is enough to run over $0\leq t_{n-2} \leq a_{n-1} - cs_{n-2}$. The second inequality is clear, because otherwise $\gchoose{a_{n-1} - cs_{n-2}  }{ t_{n-2} }=0$ by Definition~\ref{modifiedbinomialcoeff}. So we want to show that 
\begin{equation}\label{zeroeq701}
 \sum_{ t_0,t_1,\cdots,t_{n-1}} \left[ \left( \prod_{i=0}^{n-1} \left[\begin{array}{c}{{a_{i+1} - cs_i}} \\{t_i} \end{array}\right]   \right) {x_1}^{cs_{n-1}} x_2^{c(a_{n} - s_{n})}    \right] =0,
\end{equation}
where the summation runs over all integers $ t_0,...,t_{n-1}$ satisfying 
\begin{equation}\label{cond503}\left\{
\begin{array}{l} 0\leq t_i \leq a_{i+1} - cs_i \, (0\leq i\leq n-3),\\
s_{n-1} a_{n-2} -s_{n-2} a_{n-1}  \geq 0,  \\
t_{n-2}\leq a_{n-1} - cs_{n-2}<0, \text{ and }\\
s_{n} a_{n-1} -s_{n-1} a_{n}  \geq 0.      
 \end{array} \right.\end{equation}

To do this, we will show that $a_n - cs_{n-1} <0$. Suppose to the contrary that  $a_n - cs_{n-1} \geq 0$.
First of all, we have
\begin{equation}\label{eq070301}\aligned
&a_{n-3}s_{n-2} - a_{n-2}(a_{n-1}-s_{n-1}) \\
&=ca_{n-2}s_{n-2} - a_{n-1}s_{n-2}- a_{n-2}(a_{n-1}-s_{n-1})\\
&>a_{n-2}a_{n-1}- a_{n-1}s_{n-2} - a_{n-2}(a_{n-1}-s_{n-1}) \,\,\,\,\, \text{ since }a_{n-1}-cs_{n-2}<0\\
&=s_{n-1} a_{n-2} -s_{n-2} a_{n-1}\underset{\text{by }(\ref{cond503})}\geq 0.
\endaligned\end{equation}
Then \begin{equation}\label{eq070302}\aligned
& a_{n-2}s_{n-1} - a_{n-1}s_{n-2} \\
&\underset{\text{by }(\ref{eq070301})}< a_{n-2}s_{n-1} - a_{n-1}\frac{a_{n-2}}{a_{n-3}}(a_{n-1}-s_{n-1}) \\
&=a_{n-2}\left( a_{n-1} - \left(1+\frac{a_{n-1}}{a_{n-3}}\right)(a_{n-1}-s_{n-1})   \right)\\
&=a_{n-2}\left( a_{n-1} - \left(1+\frac{a_{n-1}}{a_{n-3}}\right)\frac{a_{n-2}+a_n-cs_{n-1}}{c}   \right)\\
&\leq a_{n-2}\left( a_{n-1} - \frac{a_{n-3}+a_{n-1}}{a_{n-3}}\frac{a_{n-2}}{c}   \right) \text{ since }a_n - cs_{n-1} \geq 0\,\,\,\,\,\\
&= a_{n-2}\left( a_{n-1} - \frac{a_{n-2}^2}{a_{n-3}} \right)\\
&\underset{\text{by }(\ref{negone})}=-\frac{a_{n-2}}{a_{n-3}} <0,
\endaligned\end{equation}
which contradicts $s_{n-1} a_{n-2} -s_{n-2} a_{n-1}\geq 0$. Hence \begin{equation}\label{ancsnminus1leq0}a_n - cs_{n-1} <0.\end{equation}

Next we show that $s_{n-2}>a_n-s_n$. Suppose to the contrary that $s_{n-2}\leq a_n-s_n$. Then
$$\aligned &a_{n-1}-cs_{n-2}  \geq a_{n-1}-c(a_n-s_n) \underset{\text{by }(\ref{cond503})}\geq a_{n-1}-c \frac{(a_{n-1}-s_{n-1}) a_n}{a_{n-1}} \\
&\underset{\text{by }(\ref{negone})}= \frac{a_n a_{n-2}+1  }{a_{n-1}}-c \frac{(a_{n-1}-s_{n-1}) a_n}{a_{n-1}}=\frac{a_n}{a_{n-1}}(a_{n-2}-c(a_{n-1}-s_{n-1}))+\frac{1}{a_{n-1}}\\
&=\frac{a_n}{a_{n-1}}(cs_{n-1}-a_n)+\frac{1}{a_{n-1}}\underset{\text{by }(\ref{ancsnminus1leq0})}>0,
\endaligned$$
which contradicts $a_{n-1}-cs_{n-2}<0$ in (\ref{cond503}). Thus $s_{n-2}>a_n-s_n$, so we have $$a_n-cs_{n-1}< s_n+s_{n-2}- cs_{n-1}=t_{n-1},$$which gives $\gchoose{a_n-cs_{n-1}}{t_{n-1}}=0.$ Therefore, $(\ref{zeroeq701})=0.$ 

So far we have proved that  $$x_{n+1}=x_1^{-a_n} x_2^{-a_{n-1}} \sum_{ t_0,t_1,\cdots,t_{n-1}} \left[ \left( \prod_{i=0}^{n-1} \left[\begin{array}{c}{{a_{i+1} - cs_i    } } \\{t_i} \end{array}\right]   \right) {x_1}^{cs_{n-1}} x_2^{c(a_{n} - s_{n})}    \right],$$where the summation runs over all integers $ t_0,...,t_{n-1}$ satisfying 
\begin{equation}\label{cond504}\left\{
\begin{array}{l} 0\leq t_i \leq a_{i+1} - cs_i \, (0\leq i\leq n-2),\\
s_{n-1} a_{n-2} -s_{n-2} a_{n-1}  \geq 0,  \text{ and }\\
s_{n} a_{n-1} -s_{n-1} a_{n}  \geq 0.      
 \end{array} \right.\end{equation}
But we do not have to include $s_{n-1} a_{n-2} -s_{n-2} a_{n-1}  \geq 0$ in (\ref{cond504}), because $0\leq t_i \leq a_{i+1} - cs_i \, (0\leq i\leq n-2)$ imply $s_{n-1} a_{n-2} -s_{n-2} a_{n-1}  \geq 0$ as follows.
$$\aligned &s_{n-1} a_{n-2} -s_{n-2} a_{n-1} = (cs_{n-2}-s_{n-3}+t_{n-2}) a_{n-2} -s_{n-2} a_{n-1}\\ & =  (s_{n-2}a_{n-3}-s_{n-3}a_{n-2}) + t_{n-2}a_{n-2}=\cdots= (s_{2}a_{1}-s_{1}a_{2}) +\sum_{i=2}^{n-2} t_i a_i=0+\sum_{i=2}^{n-2} t_i a_i\geq 0.\endaligned$$
This completes the proof modulo Proposition~\ref{eanegineqcor}.
\end{proof}

\begin{prop}\label{eanegineqcor}
Fix four integers $c (\geq 1), n (\geq 3), e_1$ and $e_2$ satisfying $e_2 a_{n-1} -e_1 a_{n-2}<0$. Then
\begin{equation}\label{zeroformula}\aligned
&\sum_{ t_0,t_1,\cdots,t_{n-4}} \left[ \left( \prod_{i=0}^{n-4} \left[\begin{array}{c}{{a_{i+1} - cs_i    } } \\{t_i} \end{array}\right]   \right) \right.\\
&\,\,\,\,\,\,\,\,\,\,\,\,\,\,\,\,\,\,\,\,\,\,\,\times \left.\gchoose{a_{n-2} - cs_{n-3} }{a_{n-2} - cs_{n-3}-e_2+s_{n-4}}\gchoose{-a_{n-3} + c e_2 }{  -a_{n-3} + c e_2-e_1+s_{n-3} } \right]=0,
\endaligned\end{equation}where  the summation runs over all integers $t_0,\cdots,t_{n-j-4}$ satisfying $$
0\leq t_i \leq a_{i+1}-cs_i \,\, (0\leq i\leq n-4).   
$$
\end{prop}

This is a consequence of \cite[Proposition 3.5]{SZ} in the paper by Sherman and Zelevinsky. One also may give a geometric proof. Actually one can show that $e_2 a_{n-1} -e_1 a_{n-2}<0$ implies $\langle (e_1,e_2), (a_{n-1}-e_1, a_{n-2}-e_2)  \rangle <0,$ where $\langle \cdot , \cdot \rangle$ is the Euler inner product (for instance, see \cite{S}). Then the assertion follows from a result of Schofield \cite[Section 3]{S}, which says that $$\dim \text{Gr}_{(e_1,e_2)} M(n)=\langle (e_1,e_2), (a_{n-1}-e_1, a_{n-2}-e_2)  \rangle.$$ Hence if $\langle (e_1,e_2), (a_{n-1}-e_1, a_{n-2}-e_2)  \rangle <0$ then $\text{Gr}_{(e_1,e_2)} M(n)$ is empty, so its Euler characteristic $\chi(\text{Gr}_{(e_1,e_2)} M(n))$ is obviously zero, which is equivalent to $(\ref{zeroformula})=0$ by \cite[Theorem 3.2 and (3.5)]{CZ}. 

However we will give a different proof, because we want to keep the exposition self-contained. Before we give the proof, we need some lemmas.

\begin{lem}\label{genvander}
Let $A,B,q, m$ be $($possibly negative$)$ integers with $A+B\geq q\geq 0$. Let $P(w)\in \mathbb{Q}[w]$ be any polynomial of $w$ of degree $q$. Then
$$
\sum_{w\in\mathbb{Z}} P(w) \left[\begin{array}{c}A \\w \end{array}\right]  \left[\begin{array}{c}B \\m-w \end{array}\right] =\sum_{w\in\mathbb{Z}} P(w) \left[\begin{array}{c}A \\A-w \end{array}\right]  \left[\begin{array}{c}B \\B-m+w \end{array}\right].
$$
\end{lem}

\begin{proof}
Since any polynomial of $w$ of degree $q$ is a  $\mathbb{Q}$-linear combination of  $\prod_{i=0}^{p-1}(w-i)$ $(0\leq p\leq q)$, it is enough to show that for any $p$ $(0\leq p\leq q)$,  we have
$$
\sum_{w\in\mathbb{Z}}\left( \prod_{i=0}^{p-1}(w-i)  \right)\left[\begin{array}{c}A \\w \end{array}\right]  \left[\begin{array}{c}B \\m-w \end{array}\right]  =\sum_{w\in\mathbb{Z}}\left( \prod_{i=0}^{p-1}(w-i)  \right)\left[\begin{array}{c}A \\ A-w \end{array}\right]  \left[\begin{array}{c}B \\B-m+w \end{array}\right] .
$$
If $A,B\geq 0$ then the equality is trivial. So we assume that either $A<0$ or $B<0$.  Without loss of generality, we assume $B<0$. Since $A+B\geq q$, we have $A\geq q$ hence $A\geq p$. Then
$$\aligned
 &\sum_{w\in\mathbb{Z}} \prod_{i=0}^{p-1}(w-i) \left[\begin{array}{c}A \\w \end{array}\right]  \left[\begin{array}{c}B \\m-w \end{array}\right] =\sum_{w\in\mathbb{Z}} \prod_{i=0}^{p-1}(w-i) \left[\begin{array}{c}A \\A-w \end{array}\right]  \left[\begin{array}{c}B \\m-w \end{array}\right] \,\,\,  (\text{since }A\geq 0) \\
& =\sum_{w\in\mathbb{Z}} \prod_{i=0}^{p-1}(w-i)  \prod_{i=0}^{w-1}\frac{A-i}{w-i}  \left[\begin{array}{c}B \\m-w \end{array}\right]   = \sum_{w\in\mathbb{Z}} \prod_{i=0}^{p-1}(A-i)  \prod_{i=p}^{w-1}\frac{A-i}{w-i}  \left[\begin{array}{c}B \\m-w \end{array}\right] \\
&   = \sum_{w\in\mathbb{Z}} \prod_{i=0}^{p-1}(A-i)  \left[\begin{array}{c}A-p \\w-p \end{array}\right] \left[\begin{array}{c}B \\m-w \end{array}\right] = \prod_{i=0}^{p-1}(A-i) \left[\begin{array}{c}A+B-p \\m-p \end{array}\right]\\
& = \prod_{i=0}^{p-1}(A-i) \left[\begin{array}{c}A+B-p \\A+B-m \end{array}\right] \,\,\,  (\text{since }A+B-p\geq 0)  \\
&   = \sum_{w\in\mathbb{Z}} \prod_{i=0}^{p-1}(A-i)  \left[\begin{array}{c}A-p \\A-w \end{array}\right] \left[\begin{array}{c}B \\B-m+w \end{array}\right]\\
&   = \sum_{w\in\mathbb{Z}} \prod_{i=0}^{p-1}(A-i)  \left[\begin{array}{c}A-p \\w-p \end{array}\right] \left[\begin{array}{c}B \\B-m+w \end{array}\right]  \,\,\,  (\text{since }A-p\geq 0)  \\
&=\sum_{w\in\mathbb{Z}} \prod_{i=0}^{p-1}(w-i) \left[\begin{array}{c}A \\A-w \end{array}\right]  \left[\begin{array}{c}B \\B-m+w \end{array}\right].
\endaligned$$
\end{proof}

\begin{lem}\label{importantprop} Fix four integers $c (\geq 1), n (\geq 3), e_1$ and $e_2$. Let $w_{n-2}=0$. For any $-1\leq j
\leq n-4$, define $f(j)$ by
$$\aligned
&f(j)=\sum_{ t_0,t_1,\cdots,t_{j}}\sum_{w_1,\cdots,w_{n-j-4}} \left[ \left( \prod_{i=0}^{j} \left[\begin{array}{c}{{a_{i+1} - cs_i    } } \\{t_i} \end{array}\right]   \right) \right.\\
&\,\,\,\,\,\,\,\,\,\,\,\,\,\times \gchoose{a_{j+2} - cs_{j+1} }{a_{j+2} - cs_{j+1}+s_{j}-(e_2a_{n-2-j}-e_1a_{n-3-j}-v_{n-3-j})}\\
&\,\,\,\,\,\,\,\,\,\,\,\,\,\times \gchoose{-a_{j+1} + c (e_2a_{n-2-j}-e_1a_{n-3-j}-v_{n-3-j}) }{ s_{j+1} -a_{j+1} + e_2a_{n-1-j}-e_1a_{n-2-j}-v_{n-2-j}+w_{n-3-j}}  \\
&  \,\,\,\,\,\,\,\,\,\,\,\,\left.\times \prod_{i=j+2}^{n-3}\gchoose{-a_i+c(e_2 a_{n-i-1} -e_1 a_{n-i-2}-v_{n-i-2} ) }{ -a_i+c(e_2 a_{n-i-1} -e_1 a_{n-i-2}-v_{n-i-2} )-w_{n-i-2}   }\right], \endaligned$$
where  $$v_i=\sum_{j=1}^{i-1} a_{i-j+1}w_j,$$
and the summations run over all integers $t_0,\cdots,t_j, w_1,\cdots,w_{n-j-4}$ satisfying $$
a_{i+1}-cs_i\geq 0 \, (0\leq i\leq j)    \,\,\,\text{ and } \,\,\,w_i\geq 0 \, (1\leq n-j-4).
$$
Then $f(-1)=f(0)=\cdots=f(n-4)$.
\end{lem}

\begin{proof}
This is essentially a change of variables, together with the help of Lemma~\ref{genvander}.   We frequently  use \begin{equation}\label{local1701}\aligned
&a_1=0, \,\, a_2=1, \,\, a_i=ca_{i-1} -a_{i-2},\\
&s_i=\sum_{j=1}^{i-1} a_{i-j+1}t_j=\sum_{j=1}^{i-1} (ca_{i-j}-a_{i-j-1})t_j=cs_{i-1}-s_{i-2}+t_{i-1}, \text{ and }\\
&v_i=\sum_{j=1}^{i-1} a_{i-j+1}w_j=\sum_{j=1}^{i-1} (ca_{i-j}-a_{i-j-1})w_j=cv_{i-1}-v_{i-2}+w_{i-1}.     \endaligned\end{equation} 

We will give a detailed proof for $f(n-4)=f(n-5)$. The rest of the equalities can be obtained similarly.\tiny{
$$\aligned &f(n-4)=\sum_{ t_0,t_1,\cdots,t_{n-5}}\left[\left(\prod_{i=0}^{n-5} \gchoose{a_{i+1} - cs_i  }{t_{i}}\right)\right.\\ 
&\,\,\,\,\,\,\,\,\,\,\,\,\,\,\,\,\,\,\,\,\,\,\,\,\,\,\,\,\,\,\,\left.\times\sum_{t_{n-4}\in\mathbb{Z}}\gchoose{a_{n-3} - cs_{n-4} }{t_{n-4}}\gchoose{a_{n-2} - cs_{n-3} }{a_{n-2} - cs_{n-3}-e_2+s_{n-4}}\gchoose{-a_{n-3} + c e_2 }{  -a_{n-3} + c e_2-e_1+s_{n-3} }\right].\endaligned$$}\normalsize{Since $a_{n-3} - cs_{n-4}\geq 0$, we have}\tiny{
$$\aligned &f(n-4)=\sum_{ t_0,t_1,\cdots,t_{n-5}}\left[\left(\prod_{i=0}^{n-5} \gchoose{a_{i+1} - cs_i  }{t_{i}}\right)\right.\\ 
&\,\,\,\,\,\,\,\,\,\,\,\,\,\,\,\,\,\,\,\,\,\,\,\,\,\,\,\,\,\,\,\left.\times\sum_{t_{n-4}\in\mathbb{Z}}\gchoose{a_{n-3} - cs_{n-4} }{a_{n-3} - cs_{n-4}-t_{n-4}}\gchoose{a_{n-2} - cs_{n-3} }{a_{n-2} - cs_{n-3}-e_2+s_{n-4}}\gchoose{-a_{n-3} + c e_2 }{  -a_{n-3} + c e_2-e_1+s_{n-3} }\right].\endaligned$$}\normalsize{Substituting $a_{n-3}-cs_{n-4}+s_{n-5}-(ce_2-e_1)+w_1$ into $t_{n-4}$, we get}\tiny{$$\aligned &f(n-4)=\sum_{ t_0,t_1,\cdots,t_{n-5}}\left[\left(\prod_{i=0}^{n-5} \gchoose{a_{i+1} - cs_i  }{t_{i}}\right)\right.\\ 
&\,\,\,\,\,\,\,\,\,\,\,\,\,\,\,\,\,\,\,\,\,\,\,\,\,\,\,\,\,\,\,\left.\times\sum_{w_{1}\in\mathbb{Z}}\gchoose{a_{n-3} - cs_{n-4} }{-s_{n-5}+ce_2-e_1-w_1}\gchoose{-a_{n-4}+c(ce_2-e_1) - cw_1 }{-a_{n-4}+c(ce_2-e_1) - cw_1-e_2+s_{n-4}}\gchoose{-a_{n-3} + c e_2 }{ w_1 }\right].\endaligned$$}\normalsize{Here }\tiny{$\gchoose{-a_{n-4}+c(ce_2-e_1) - cw_1 }{-a_{n-4}+c(ce_2-e_1) - cw_1-e_2+s_{n-4}}$}\normalsize{ (if nonzero) can be regarded as a polynomial of $w_1$ of degree $e_2-s_{n-4}\geq 0$. Since $$0\leq e_2-s_{n-4}\leq (a_{n-3} - cs_{n-4}) +(-a_{n-3} + c e_2),$$ we can apply Lemma~\ref{genvander}. Then we obtain}\tiny{
$$\aligned &f(n-4)\\
&=\sum_{ t_0,t_1,\cdots,t_{n-5}}\left[\left(\prod_{i=0}^{n-5} \gchoose{a_{i+1} - cs_i  }{t_{i}}\right)\right.\\ 
&\,\,\,\,\,\,\,\,\,\,\,\,\,\,\left.\times\sum_{w_{1}\in\mathbb{Z}}\gchoose{a_{n-3} - cs_{n-4} }{a_{n-3} - cs_{n-4}+s_{n-5}-(ce_2-e_1)+w_1}\gchoose{-a_{n-4}+c(ce_2-e_1) - cw_1 }{-a_{n-4}+c(ce_2-e_1) - cw_1-e_2+s_{n-4}}\gchoose{-a_{n-3} + c e_2 }{-a_{n-3} + c e_2- w_1 }\right]\\
&=\sum_{ t_0,t_1,\cdots,t_{n-5}}\left[\left(\prod_{i=0}^{n-5} \gchoose{a_{i+1} - cs_i  }{t_{i}}\right)\right.\\ 
&\,\,\,\,\,\,\,\,\,\,\,\,\,\,\,\,\,\,\,\,\,\,\,\,\,\,\,\,\,\,\,\,\,\,\,\,\,\,\,\,\,\,\,\,\,\,\,\times\sum_{w_{1}\in\mathbb{Z}}\gchoose{a_{n-3} - cs_{n-4} }{a_{n-3} - cs_{n-4}+s_{n-5}-(e_2 a_3-e_1 a_2- v_2)} \gchoose{-a_{n-4}+c(e_2 a_3-e_1 a_2- v_2) }{s_{n-4}-a_{n-4}+e_2a_4-e_1a_3-v_3+w_2 }\\
&\,\,\,\,\,\,\,\,\,\,\,\,\,\,\,\,\,\,\,\,\,\,\,\,\,\,\,\,\,\,\,\,\,\,\,\,\,\,\,\,\,\,\,\,\,\,\,\left.\times\gchoose{-a_{n-3}+c(e_2 a_{2} -e_1 a_{1}-v_{1} ) }{-a_{n-3}+c(e_2 a_{2} -e_1 a_{1}-v_{1} )-w_{1}   }\right],\endaligned $$}\normalsize{where we have used (\ref{local1701}). Since }\tiny{$\gchoose{-a_{n-3}+c(e_2 a_{2} -e_1 a_{1}-v_{1} ) }{-a_{n-3}+c(e_2 a_{2} -e_1 a_{1}-v_{1} )-w_{1}   }$}\normalsize{$=0$ for $w_1<0$, we actually have }\tiny{$$\aligned & f(n-4)=\sum_{ t_0,t_1,\cdots,t_{n-5}}\left[\left(\prod_{i=0}^{n-5} \gchoose{a_{i+1} - cs_i  }{t_{i}}\right)\right.\\ 
&\,\,\,\,\,\,\,\,\,\,\,\,\,\,\,\,\,\,\,\,\,\,\,\,\,\,\,\,\,\,\,\,\,\,\,\,\,\,\,\,\,\,\,\,\,\,\,\,\,\,\,\,\,\,\,\,\,\,\,\times\sum_{w_{1}\geq 0}\gchoose{a_{n-3} - cs_{n-4} }{a_{n-3} - cs_{n-4}+s_{n-5}-(e_2 a_3-e_1 a_2- v_2)} \gchoose{-a_{n-4}+c(e_2 a_3-e_1 a_2- v_2) }{s_{n-4}-a_{n-4}+e_2a_4-e_1a_3-v_3+w_2 }\\
&\,\,\,\,\,\,\,\,\,\,\,\,\,\,\,\,\,\,\,\,\,\,\,\,\,\,\,\,\,\,\,\,\,\,\,\,\,\,\,\,\,\,\,\,\,\,\,\,\,\,\,\,\,\,\,\,\,\,\,\left.\times\gchoose{-a_{n-3}+c(e_2 a_{2} -e_1 a_{1}-v_{1} ) }{-a_{n-3}+c(e_2 a_{2} -e_1 a_{1}-v_{1} )-w_{1}   }\right],\endaligned $$}\normalsize{which is equal to $f(n-5)$.

Again since $a_{n-4} - cs_{n-5}\geq 0$, we have}
\tiny{$$\aligned & f(n-5)=\sum_{ t_0,t_1,\cdots,t_{n-6}}\sum_{w_{1}\geq 0}\left[\left(\prod_{i=0}^{n-6} \gchoose{a_{i+1} - cs_i  }{t_{i}}\right)\right.\\ 
&\,\,\,\,\,\times\gchoose{a_{n-4} - cs_{n-5}   }{a_{n-4} - cs_{n-5}-t_{n-5}  }  \gchoose{a_{n-3} - cs_{n-4} }{a_{n-3} - cs_{n-4}+s_{n-5}-(e_2 a_3-e_1 a_2- v_2)} \gchoose{-a_{n-4}+c(e_2 a_3-e_1 a_2- v_2) }{s_{n-4}-a_{n-4}+e_2a_4-e_1a_3-v_3+w_2 }\\
&\,\,\,\,\,\left.\times\gchoose{-a_{n-3}+c(e_2 a_{2} -e_1 a_{1}-v_{1} ) }{-a_{n-3}+c(e_2 a_{2} -e_1 a_{1}-v_{1} )-w_{1}   }\right].\endaligned $$}\normalsize{In the same manner as above, i.e. by substituting $a_{n-4}-cs_{n-5}+s_{n-6}-(e_2a_4-e_1a_3)+v_3$ into $t_{n-5}$ and then applying Lemma~\ref{genvander}, it is not hard to show that
$$f(n-5)=f(n-6).$$

Repeating this process, we eventually obtain the desired equalities.}
\end{proof}

\begin{proof}[Proof of Proposition~\ref{eanegineqcor}]
By Lemma~\ref{importantprop}, the left-hand side of (\ref{zeroformula}), which is $f(n-4)$, is equal to $f(-1)$. Since $a_1=0$, the first modified binomial coefficient in $f(-1)$ is equal to $$\gchoose{0}{-(e_2a_{n-1}-e_1a_{n-2}-\sum_{j=1}^{n-3} a_{n-1-j}w_j)}.$$ Here 
$\sum_{j=1}^{n-3} a_{n-1-j}w_j\geq 0$ since $w_j\geq 0$. Therefore, if $e_2a_{n-1}-e_1a_{n-2}<0$ then  $$\gchoose{0 }{-(e_2a_{n-1}-e_1a_{n-2}-\sum_{j=1}^{n-3} a_{n-1-j}w_j)}=0$$ for any $w_j\geq 0$, which gives $f(-1)=0$. This completes the proof.
\end{proof}


\begin{proof}[Proof of Corollary~\ref{maincor}]
Corollary~\ref{maincor} is an immediate consequence of Theorem~\ref{mainthm} thanks to a result of Caldero and  Zelevinsky \cite[Theorem 3.2 and (3.5)]{CZ}. If $e_2 a_{n-1} -e_1 a_{n-2}<0$ then following from the discussion after  Proposition~\ref{eanegineqcor}, we have $(\ref{submainformula})=0=\chi(\text{Gr}_{(e_1,e_2)} M(n))$.  
\end{proof}

\begin{proof}[Proof of Corollary~\ref{maincor2}]
By (\ref{cond512}), all the modified binomial coefficients except for the last one in (\ref{submainformula}) are non-negative. If $e_2\geq \frac{a_{n-3}}{c}$ then the last one also becomes non-negative. Therefore, Corollary~\ref{maincor} implies that $\chi(\text{Gr}_{(e_1,e_2)} M(n))$ is non-negative.
\end{proof}

\end{document}